\documentclass{article}
\usepackage{import}

\usepackage[english]{babel}  % Englisch 
\usepackage[dvipsnames]{xcolor} %mehr farben
\usepackage{amsmath}  % Allgemeiner Formelsatz
\usepackage{amssymb}  % Noch ein paar mathematische Symbole
\usepackage{amsthm}   % Mehr Optionen für theorem-Umgebungen
\usepackage[utf8]{inputenc}  % universelle Eingabecodierung
\usepackage{graphicx}
\usepackage{color}
\usepackage[left=2.75cm,right=2.75cm,top=3cm,bottom=3.5cm]{geometry}
\usepackage{bbm,mathabx}
\usepackage{todonotes}
\usepackage{amsfonts,nccbbb,xfrac,mathtools}
\usepackage[shortlabels]{enumitem}
\usepackage{fancyhdr}
\usepackage{autobreak}
\usepackage{nicefrac, xfrac}

\usepackage[section]{placeins}
\usepackage{pgfplots}
\pgfplotsset{compat=1.18,width = 8.5cm}

\usepackage{mathtools}
\usepackage{physics}
\usepackage{esint}
\usepackage{verbatim}
\usepackage[nottoc]{tocbibind}
\usepackage{enumitem}
\usepackage{setspace}
\usepackage{subcaption}
\usepackage{float}
\usepackage{ragged2e}
\usepackage{aligned-overset}
\allowdisplaybreaks
\usepackage[pdfencoding=auto, psdextra]{hyperref}

\newtheorem{thm}{Theorem}%[section]
\newtheorem*{thm*}{Theorem}
\newtheorem{prop}[thm]{Proposition}

\newtheorem*{cor*}{Corollar}
\newtheorem{lem}[thm]{Lemma}
\newtheorem*{lem*}{Lemma}
\theoremstyle{definition}

\newtheorem*{example*}{Example}

\newtheorem*{rmk*}{Remark}
\newtheorem{defin/thm}[thm]{Definition/Theorem}
\newtheorem{defin/rmk}[thm]{Definition/Remark}
\newtheorem{nota/rmk}[thm]{Notation/Remark}
\newtheorem*{nota/rmk*}{Notation/Remark}

\newtheorem*{motiv*}{Motivation}

\newtheorem*{nota*}{Notation}

\newtheorem*{question*}{Frage}

\definecolor{Red}{rgb}{1,0,0}

\newcommand{\EE}{\mathbb{E}}

\newcommand{\RR}{\mathbb{R}}

\newcommand{\NN}{\mathbb{N}}

\begin{document}

\title{A large deviation principle for the Gaussian beta ensemble in the high temperature regime}

\author{Helene Götz, Jan Nagel}

\maketitle

\begin{abstract}
We show a large deviation principle for the weighted spectral measure of the Gaussian beta ensemble in the high temperature regime, when $n\beta_n\to 2c \in [0,\infty)$. 
The resulting rate function has a novel and interesting structure: it is finite only for discrete measures and moreover, in the very high-temperature regime (when $c=0$), additional terms may arise that penalize each support point beyond the first.

\end{abstract}

%%%%%%%%%%%%%%%%%%%%%%%%%%%%%%%%%%%%%%%%%%%%%%%%%%%%%%%%%%%%%%%%%%%%%%%%%
%%%%%%%%%%%%%%%%%%%%%%%%%%%%%%%%%%%%%%%%%%%%%%%%%%%%%%%%%%%%%%%%%%%%%%%%%



	\maketitle
\section{Introduction}
The Gaussian ensemble, also known as Wigner ensemble or Hermite ensemble, is one of the most classical ensembles in random matrix theory.
Its eigenvalues $\lambda_1,...,\lambda_n$ have the joint Lebesgue density
\begin{align}
	\label{eq:DensityGaussnotscaled}
	f_{\beta,n}(\lambda_1,...,\lambda_n) = Z_{\beta}^G \prod_{1 \leq i<j \leq n} |\lambda_i-\lambda_j|^{\beta}  e^{-\sum_{i=1}^n \frac{n\beta \lambda_i^2}{2}}
\end{align}
with $\beta>0$. When $\beta\in\{1,2,4\}$, full matrix models exist with entries real, complex, or quaternion Gaussian random variables, respectively, with eigenvalue density \eqref{eq:DensityGaussnotscaled}. For general $\beta>0$, \eqref{eq:DensityGaussnotscaled} is the particle density of a log-gas at inverse temperature $\beta$, and the eigenvalue density of a tridiagonal matrix model. We refer to \cite{mehta2004random,anderson2010introduction,forrester2010log,tao2023topics} for extensive informations on the Gaussian ensemble.   

Many limit theorems deal with the empirical eigenvalue measure $\bar\mu_n = \frac{1}{n} \sum_{k=1}^n\delta_{\lambda_k}$ of a matrix $X_n$ of the Gaussian ensemble, for which Wigner showed the convergence to the semicircle law \cite{wigner1958distribution}. We are interested in large deviation principles, which quantify the exponentially fast convergence, and have been proven for the sequence of empirical eigenvalue measures by \cite{arous1997large}. In \cite{gamboa2011large,gamboa2016sumrules}, large deviation principles have been shown for the weighted spectral measure of a matrix $X_n$ of the Gaussian ensemble, whose $k$-th moment is given by $(X_n^k)_{1,1}$, and where the weights $\tfrac{1}{n}$ of the empirical measure are replaced by Dirichlet distributed weights.   
The above results hold under the standard scaling, where $\beta>0$ is fixed and $n$ tends to infinity. 

In the high-temperature regime, $\beta=\beta_n$ depends on $n$ and tends to zero, such that $n\beta_n\to 2c\in(0,\infty)$. 
First studied in the context of Dyson Brownian motion \cite{cepa1997diffusing,allez2012invariant}, by now many limit theorems for fluctuations of the empirical eigenvalue measure have been transferred to this setting \cite{benaych2015poisson,nakano2018gaussian,dworaczek2024clt}.
A large deviation principle follows from \cite{garcia2019large,liu2020large}, see also the discussion in \cite{nakano2020poisson}, and large deviations of extreme eigenvalues have been studied by   
\cite{pakzad2020large,guera2025large,guionnet2022large}. 

We are interested in the large deviation behavior of the weighted spectral measure in the high temperature regime. While the limit of the empirical eigenvalue measure is a deterministic measure with full support on $\mathbb R$, the weighted spectral measure converges in distribution to a nondeterministic Dirichlet process \cite{nakano2026spectral}. 
Therefore, we rescale the eigenvalues by $\sqrt \beta_n$ and consider the measure
\begin{align} \label{def:spectralmeasure}
\mu_n = \sum_{k=1}^n w_k \delta_{\sqrt{\beta_n} \lambda_k} . 
\end{align}
Omitting the factor $\beta_n$ results precisely in the measure considered in \cite{nakano2026spectral}. The weights $(w_1,\dots ,w_n)$ are independent of the eigenvalues and Dirichlet distributed on the unit simplex with uniform parameter $\beta_n/2$. The measure $\mu_n$ is a random element of the space $\mathcal M_1(\mathbb R)$ of probability measures on the real line, equipped with the weak topology and its corresponding $\sigma$-algebra. 

Throughout this paper, we make the following assumptions: 
\begin{align}\label{assumptions}
\lim_{n\to \infty} n\beta_n =2c\in [0,\infty),\qquad \lim_{n\to \infty} -\frac{1}{n}\log(\beta_n) = \xi\in [0,\infty] . 
\end{align}
When the first limit exists with $c\in(0,\infty)$, the second condition is automatically satisfied with $\xi=0$. Therefore, the second condition is only relevant in the \emph{very high} temperature regime, when $n\beta_n\to 0$. In this case we have to distinguish between a \emph{subexponential} regime when $\xi=0$, an \emph{exponential} regime with $\xi\in(0,\infty)$ and a \emph{superexponential} regime with $\xi=\infty$. 
We then have the following large deviation principle. 

\begin{thm}\label{thm:mainLDP}
Suppose the limits in \eqref{assumptions} exist, then the sequence of weighted spectral measures $\mu_n$ as defined in \eqref{def:spectralmeasure} satisfies a large deviation principle with speed $n$ and good rate function $\mathcal I$. For $\mu$ a measure with countable support $\operatorname{supp}(\mu)$, we have
\begin{align*}
\mathcal I(\mu) =  (|\operatorname{supp}(\mu)|-1)\cdot \xi  + \frac{1}{2}\sum_{\lambda \in \operatorname{supp}(\mu)} \lambda^2 .
\end{align*} 
In the case of infinitely many support points and $\xi=0$, or if $\mu$ has a single support point and $\xi = \infty$, we interpret $\infty\cdot 0$ or $0\cdot \infty$, respectively, as $0$. If the support of $\mu$ is uncountable, we have $\mathcal I(\mu)=+\infty$.  
\end{thm}

Let us remark that the rate function differs substantially from rate functions in other large deviation principles as in \cite{arous1997large,liu2020large}, which is only finite for absolutely continuous measures, or in \cite{gamboa2016sumrules}, which contains a Kullback-Leibler divergence. The value $\mathcal I(\mu)$ of the good rate function in Theorem \ref{thm:mainLDP} is zero if and only if $\mu=\delta_0$ is the Dirac measure in zero. As a consequence of Theorem \ref{thm:mainLDP}, we therefore have the exponentially fast convergence of $\mu_n$ to $\delta_0$. In the superexponential regime when $\xi=\infty$, the value $\mathcal I(\mu)$ of the rate is only finite if $\mu$ is a measure with a single support point. In the exponential regime, $\mathcal I(\mu)$ is finite only for finitely supported measures, with an additional cost of $\xi$ for every support point after the first. 

Only when $\xi=0$, as in the classical high-temperature regime with $c\in (0,\infty)$, we may have $\mathcal I(\mu)<\infty$ for measures with infinite support.  In this case, the rate function may be rewritten as follows: when $\mu$ is the projected spectral measure of a bounded self-adjoint operator $H_\mu$ with some cyclic vector $e$ (see \eqref{eq:spectralmeasuremoments}), we have
\begin{align} \label{sumrule}
\mathcal I(\mu) = \tfrac{1}{2} ||H_\mu||_{\mathrm{HS}}^2 ,
\end{align}
where $||\cdot ||_{\mathrm{HS}}$ is the Hilbert-Schmidt norm, such that $\mathcal I(\mu)<\infty$ if and only if $H_\mu$ is a Hilbert-Schmidt operator. The identity \eqref{sumrule} is then an example of a \emph{sum rule}, as discussed in \cite[Chapter 1]{simon2011szego}. The close connection between large deviation principles for random matrices and sum rules from spectral theory has been established in a series of papers \cite{gamboa2011large,gamboa2016sumrules,gamboa2022sumrules,gamboa2025multicutcircle}
and it is even possible to prove such a sum rule with the help of large deviation theory. 

The strategy to prove Theorem \ref{thm:mainLDP} makes use of the tridiagonal representation of the Gaussian $\beta$-ensemble due to \cite{Dumitriu_2002}, a random tridiagonal matrix $T_n$ build from independent normal and gamma random variables, which is the Jacobi matrix of a spectral measure as in \eqref{def:spectralmeasure}. This allows for a relatively short proof, starting with LDPs for entries of $T_n$, which then can be extended to LDPs for the tridiagonal matrix $T_n$, and then transferred to the random measure $\mu_n$. These steps are carried out in the following two sections. 

We will observe in the proof that it is very unlikely for the off-diagonal entries of $T_n$ to not be close to zero. In particular, when $\xi>0$, the joint rate function for the entries of $T_n$ adds the value $\xi$ for every non-zero off-diagonal term. Since the number of support points of a spectral measure is directly linked to the number of positive off-diagonal terms in its tridiagonal Jacobi matrix, this explains the additional cost of $\xi$ in the rate function in Theorem \ref{thm:mainLDP} for every support point after the first one.   

The strong concentration on measures with few support points can also be explained with the behavior of the Dirichlet weights: when $\beta_n$ is very small, most of the mass is concentrated on one support point, it is very unlikely to observe several large values among $w_1,\dots ,w_n$. The support points $\sqrt{\beta_n}\lambda_1,\dots ,\sqrt{\beta_n}\lambda_n$ converge to 0, and large values for support points are not very likely. These statements will be made more precise in a companion paper \cite{spectralLDP}, where we study the large deviation behavior of eigenvalues and Dirichlet distributed weight vectors in the high temperature regime. 

\section{Large deviations for the tridiagonal model}

%In the following, we write $\beta$ for $\beta_n$ to simplify notation. 
The spectral measure defined in \eqref{def:spectralmeasure} has support points $\tilde \lambda_i = \sqrt{\beta_n}\lambda_i$, where the rescaled eigenvalues $(\tilde\lambda_1,\dots \tilde \lambda_n)$ have the Lebesgue density
\begin{align}
	\tilde Z_{\beta}^G \prod_{1 \leq i<j \leq n} |\tilde\lambda_i-\tilde\lambda_j|^{\beta_n}  e^{-\sum_{i=1}^n \frac{n \tilde\lambda_i^2}{2}} . 
\end{align}
A symmetric tridiagonal matrix with such an eigenvalue density was constructed by \cite{Dumitriu_2002}, generalizing previous work by \cite{trotter1984}, and is given as
\begin{align} \label{eq:tridiagonal}
T_n =\begin{pmatrix}
	b^{(n)}_1 & a^{(n)}_1 &  \\
	a^{(n)}_1 & b^{(n)}_2 & \ddots \\
	& \ddots & \ddots & a^{(n)}_{n-1}\\
	& & a^{(n)}_{n-1} & b^{(n)}_n
\end{pmatrix} ,
\end{align}
where $b^{(n)}_1,\dots ,b^{(n)}_n,a_1^{(n)},\dots ,a_{n-1}^{(n)}$ are independent random variables, the $a_k^{(n)}$ are positive and 
\begin{align*}
	b_k^{(n)} &\sim \mathcal{N}(0,\tfrac{1}{n}) ,  \\
	\big(a_k^{(n)}\big)^2 & \sim \Gamma\big((n-k)\tfrac{\beta_n}{2},\tfrac{1}{n}\big)  .
\end{align*}
Here the Gamma-distribution is parametrized by scale parameter, such that the expectation of $\Gamma(\alpha,\theta)$ with $\alpha,\theta >0$ is given by $\alpha\theta$. The first unit vector $e_1$ is cyclic for $T_n$ and our random spectral measure $\mu_n$ as given in \eqref{def:spectralmeasure} can then be defined as the spectral measure of the pair $(T_n,e_1)$, meaning that it is defined by the moment relation
\begin{align}\label{eq:spectralmeasuremoments}
	\int x^k \, d\mu_n(x) = \langle e_1, T_n^k e_1 \rangle.
\end{align}
Indeed, this spectral measure is given by 
\begin{align} \label{def:spectralmeasure2}
	\mu_n = \sum_{i=1}^n w_i \delta_{\tilde \lambda_i} , 
\end{align}
supported by the eigenvalues $\tilde \lambda_1,\dots ,\tilde \lambda_n$ of $T_n$ and, as proven by \cite{Dumitriu_2002}, the weights $w_1,\dots ,w_n$ are independent of the eigenvalues and follow a Dirichlet distribution on the standard simplex with homogeneous parameter $\beta_n/2$, that is, $(w_1,\dots ,w_{n-1})$ have the Lebesgue density
\begin{align} \label{eq:dirdensity}
	\frac{\Gamma(n\beta_n/2 )}{\Gamma(\beta_n/2)^n} \big[ w_1\dots w_{n-1}(1-w_1-\dots -w_{n-1})\big]^{\beta_n/2-1} 
\end{align}
on $(0,1)^{n-1}$. We may therefore study the large deviation behavior of $\mu_n$ by starting with the entries of the tridiagonal matrix $T_n$. The following lemma yields LDPs for individual entries.

\begin{lem} 
	\label{lem:LDPentries}
Suppose the sequence $(\beta_n)_n$ satisfies the assumptions in \eqref{assumptions}. 
	\begin{enumerate}
		\item[(a)] 
		Let $(X_n)_n$ be a sequence of random variables with $X_n \sim N(0,\tfrac{1}{n})$. Then $(X_n)_n$ satisfies an LDP in $\RR$ with speed $n$ and good rate function given by $\mathcal I_1(x) = x^2/2$.
		\item[(b)]
		Let $(Y_n)_n$ be a sequence of random variables with $Y_n \sim \Gamma((n-k)\tfrac{\beta_n}{2}, \tfrac{1}{n})$. Then $(Y_n)_n$ satisfies an LDP in $[0,\infty)$ with speed $n$ and a good rate function $\mathcal{I}_2$ given by 			\begin{align*}
				\mathcal{I}_2(z) = \begin{cases}
					z + \xi & \text{if } z>0, \\
					0 & \text{if } z = 0.
				\end{cases}
			\end{align*}
	\end{enumerate}
\end{lem}

\begin{proof}
	The proof of part (a) uses the Gärtner-Ellis Theorem \cite[Theorem 2.3.6]{dembo2009large}. For this consider
	\begin{align*}
		\frac{1}{n}\log\EE\exp\big(nX_n t\big)
		=  \frac{1}{n} \log\big(e^{n \frac{t^2}{2}}\big)
		= \frac{t^2}{2}.
	\end{align*}
	Since the mapping $t\mapsto \tfrac{1}{2}t^2$ is finite, differentiable on $\RR$ and steep, the good rate function is given by its Fenchel-Legendre transform, which is $\mathcal I_1$. 
	
For the gamma distributed entries, we will show a weak LDP by proving an upper and lower bound for the probability of small balls. By showing exponential tightness, we can improve the result to a strong LDP. Fix $z\geq 0$ and let $B_\varepsilon(z) = (z-\varepsilon,z+\varepsilon)$ for $\varepsilon>0$. The weak LDP follows once we show that 
\begin{align}\label{eq:weakLDP}
\lim_{\varepsilon\to 0} \limsup_{n \to \infty} \frac{1}{n} \log P\big(Y_n \in B_{\varepsilon}(z)\big)  = \lim_{\varepsilon\to 0} \liminf_{n \to \infty} \frac{1}{n} \log P\big(Y_n \in B_{\varepsilon}(z)\big) = -\mathcal I_2(z) .  
\end{align} 
	 First suppose $z>0$ and let $\varepsilon >0$ be small enough such that $z-\varepsilon >0$. We need to estimate 
\begin{align*}
\frac{1}{n} \log P\big(Y_n \in B_{\varepsilon}(z)\big) = \frac{1}{n}\log\left( \frac{n^{(n-k)\beta_n/2}}{\Gamma\big( (n-k)\beta_n/2\big)} \int_{z-\varepsilon}^{z + \varepsilon} x^{(n-k)\beta_n/2-1 } e^{-nx}\,dx\right).
\end{align*}	 
To simplify the normalization constant, we may use
\begin{align} \label{eq:normconst}
\lim_{n\to \infty} \frac{1}{n}\log \left( \frac{1}{\Gamma\big( (n-k)\beta_n/2\big)} \right) = 
\lim_{n\to \infty}\frac{1}{n}\log\left((n-k)\beta_n/2 \right) = \lim_{n\to \infty} \frac{1}{n} \log(\beta_n) = -\xi , 
\end{align}
where in the case $n\beta_n\to 0$, the first equality follows since $\lim_{x \to 0}x\Gamma(x) = 1$. We then have the lower bound
	\begin{align*}
		 \liminf_{n \to \infty} \frac{1}{n} \log P\big(Y_n \in B_{\varepsilon}(z)\big) 
		& = - \xi + \liminf_{n\to \infty} \frac{1}{n} \log \left( \int_{z-\varepsilon}^{z + \varepsilon} x^{(n-k)\beta_n/2-1 } e^{-nx}\,dx\right) \\
		&\geq -\xi 
		+ \liminf_{n\to \infty} \frac{1}{n} \log \left( \frac{(z-\varepsilon)^{(n-k)\beta_n/2}}{z+\varepsilon}\,e^{-n(z+\varepsilon)}\right)\\
		&= -\xi - (z+\varepsilon) . 
	\end{align*}
Analogously, we have
\begin{align*}
\limsup_{n \to \infty} \frac{1}{n} \log P\big(Y_n \in B_{\varepsilon}(z)\big)
\leq -\xi - (z-\varepsilon) .
\end{align*}
This implies that \eqref{eq:weakLDP} holds for $z>0$ with $\mathcal I_2(z) = z+\xi$, with $\xi $ as in \eqref{assumptions}. 
When $z=0$, we have the lower bound
	\begin{align*}
		\liminf_{n \to \infty} \frac{1}{n} \log P\big(Y_n \in B_{\varepsilon}(0)\big) 
		&=  \liminf_{n \to \infty}\frac{1}{n} \log \left( \big( (n-k)\beta_n/2\big)\int_{0}^{\varepsilon} (nx)^{(n-k)\beta_n/2-1 } e^{-nx}n\,dx \right)\\
		&=  \liminf_{n \to \infty}\frac{1}{n} \log \left( \big( (n-k)\beta_n/2\big) \int_{0}^{n\varepsilon} x^{(n-k)\beta_n/2-1 } e^{-x}\,dx \right)\\
		&\geq  \liminf_{n \to \infty}\frac{1}{n} \log \left( \big( (n-k)\beta_n/2\big) \int_{0}^{\varepsilon} x^{(n-k)\beta_n/2-1 } e^{-x}\,dx \right)\\
		&\geq  \liminf_{n \to \infty}\frac{1}{n} \log \left( \big( (n-k)\beta_n/2\big)e^{-\varepsilon}\int_{0}^{\varepsilon} x^{(n-k)\beta_n/2-1 } \,dx \right)\\
		& =   \liminf_{n \to \infty}\frac{1}{n} \log \left( \varepsilon^{(n-k)\beta_n/2 } e^{-\varepsilon}  \right)\\
		&= 0.
	\end{align*}
%	Since $\log P\big(Y_n \in B_{\varepsilon}(0)\big) \leq 0$, we have
%	\begin{align*}
%		\limsup_{\varepsilon \to 0}\lim_{n \to \infty} \frac{1}{n} \log P\big(Y_n \in B_{\varepsilon}(0)\big) \leq 0.
%	\end{align*}
This implies \eqref{eq:weakLDP} for $z=0$. The statement \eqref{eq:weakLDP} implies then that $(Y_n)_n$ satisfies the weak LDP with speed $n$ and the good rate function $\mathcal I_2$ by \cite[Theorem 4.1.11]{dembo2009large} (see also Corollary D.6 in \cite{anderson2010introduction}). It remains to show that $(Y_n)_n$ is exponentially tight, this strengthens the result to a full LDP by \cite[Lemma 1.2.18]{dembo2009large}. 

For the exponential tightness, define $K_M = [0,2M]$ for $M>0$. Since $n\beta_n$ is bounded from above, we have $x^{(n-k)\beta_n-1} \leq e^{x/2}$ on $[2nM,\infty)$ for $n$ large enough. This yields
	\begin{align*}
		\limsup_{n \to \infty} \frac{1}{n} \log P\big(Y_n \in K_M^c\big)
		&=  \limsup_{n \to \infty}\frac{1}{n} \log \left( \big( (n-k)\beta_n/2\big)\int_{2nM}^{\infty} x^{(n-k)\beta_n/2-1 } e^{-x}\,dx \right) \\
		&\leq \limsup_{n \to \infty} \frac{1}{n} \log(\beta_n) + \limsup_{n \to \infty} \frac{1}{n} \log \left(\int_{2nM}^\infty e^{x/2} e^{-x} \,dx \right) \\
		&\leq \limsup_{n \to \infty} \frac{1}{n} \log \left(\int_{2nM}^\infty e^{-x/2} \,dx \right) \\
		&= -M,
	\end{align*}
that is, the sequence $(Y_n)_n$ is exponentially tight. 
\end{proof}

Lemma \ref{lem:LDPentries} implies for any $k\geq 1$ the LDP for the sequence $(b_k^{(n)})_{n\geq 1}$ with speed $n$ and good rate function $\mathcal I_1$, where $b_k^{(n)}$ is the entry of the tridiagonal matrix \eqref{eq:tridiagonal} and we set $b_k^{(n)}=0$ if $k>n$. Similarly, the sequence $((a_k^{(n)})^2)_{n\geq 0}$ satisfies the LDP with speed $n$ and good rate function $\mathcal I_2$, setting $a_k^{(n)}=0$ for $k\geq n$. The contraction principle implies the LDP for the sequence $(a_k^{(n)})_{n\geq 0}$ in $[0,\infty)$ with good rate function given by 
\begin{align*}
\mathcal I_3(z) = \mathcal I_2(z^2) = \begin{cases}
					z^2 + \xi & \text{if } z>0, \\
					0 & \text{if } z = 0.
				\end{cases}
\end{align*}
Since the entries $b^{(n)}_1,a^{(n)}_1,b^{(n)}_2,...,a^{(n)}_{n-1},b^{(n)}_n$ are independent, Lemma \ref{lem:LDPentries} and \cite[Exercise 4.2.7]{dembo2009large} imply for any $N\geq 1$ the LDP for the sequence 
\begin{align*}
\big( (b_1^{(n)},a_1^{(n)},\dots ,b_{N}^{(n)},a_{N}^{(n)})\big)_{n\geq 1}
\end{align*}
in $(\mathbb R\times [0,\infty))^N$ with speed $n$ and good rate function defined by 
\begin{align*}
	\mathcal{I}_{4,N}(x_1,\dots ,x_{2N}) =  \sum_{k=1}^N \mathcal I_1(x_{2k-1}) + \mathcal I_3(x_{2k}) . 
\end{align*}
Let us now denote by 
\begin{align} \label{eq:recsequence}
r^{(n)} = \big(b_1^{(n)},a_1^{(n)},b_2^{(n)},a_2^{(n)},\dots,b_n^{(n)},0,\dots \big)
\end{align}
the sequence of all entries of $T_n$, extended by zeroes. Then the projective method of the Dawson-Gärtner Theorem \cite[Theorem 4.6.1]{dembo2009large} implies that $(r^{(n)})_{n\geq 1}$ satisfies an LDP, which we formulate in the following proposition. 

\begin{prop} \label{prop:projectiveLDP}
Suppose the sequence $(\beta_n)_n$ satisfies the assumptions in \eqref{assumptions}. Then the sequence $(r^{(n)})_{n\geq 1}$ as defined in \eqref{eq:recsequence} satisfies the LDP in $(\RR\times [0,\infty))^\NN$ with speed $n$ and good rate function
\begin{align*}
	\mathcal{I}_5(r)  = \sup_{N\geq 1} \mathcal{I}_{4,N}(b_1,a_1,\dots b_N ,a_N) = \frac{1}{2}\sum_{k=1}^\infty b_{k}^2 + \sum_{k \in J(r)} (a_k^2 + \xi)
\end{align*}
where $r= (b_1,a_1,b_2,a_2,\dots )$ and $J(r) := \{k \in \NN \mid a_k >0\}$.  
\end{prop}

\section{Mapping to spectral measures: proof of Theorem \ref{thm:mainLDP}}

In order to prove Theorem \ref{thm:mainLDP}, we will make use of the contraction principle to transfer the LDP in Proposition \ref{prop:projectiveLDP} to the sequence of spectral measures $\mu_n$ as defined in \eqref{def:spectralmeasure}. Some care is necessary to obtain a suitable well-defined continuous mapping from the space of Jacobi coefficients (the entries of tridiagonal matrices) to the space of probability measures. 

To begin with, we note that $\mathcal I_5(r)=\infty$ if $r$ is not a bounded sequence, so that by \cite[Lemma 4.1.5]{dembo2009large}, the LDP for the sequence $(r^{(n)})_n$ also holds in the space 
\begin{align}
\mathcal{R} := \left\{r=(b_1,a_1,\dots ) \in (\RR \times [0,\infty))^\NN \, \left| \, \sup_{k\geq 1} (|b_k|+a_k)<\infty\right. \right\} . 
\end{align}
Let us now consider the subspace $\mathcal E_\infty\subset \mathcal R$ with
\begin{align} \label{eq:Einfty}
\mathcal{E}_\infty = \big\{r=(b_1,a_1,\dots )\in \mathcal R \, \mid \, a_k>0 \text{ for all }k\geq 1 \big\} . 
\end{align}
For $r\in \mathcal E_\infty$, we define the infinite Jacobi-matrix 
\begin{align}\label{eq:tridigonalinfinite}
	\mathcal{J}_\infty(r)  = \begin{pmatrix}
		b_1 & a_1 &  \\
		a_1 & b_2 & a_2 \\
		& a_2 & \ddots & \ddots\\
		& & \ddots &  
	\end{pmatrix} .
\end{align}
Then $\mathcal J_\infty(r)$ is a self-adjoint bounded operator on $\ell^2$ with cyclic vector the first unit vector $e_1$. By the spectral theorem, there exists a unique spectral measure $\mu=\mu(r)$ of the pair $(\mathcal J_\infty(r),e_1)$, which may be defined by the moment relation
\begin{align}\label{eq:spectralmeasuremoments}
	\int x^k \, d\mu(x) = \langle e_1, \mathcal J_\infty(r)^k e_1 \rangle 
\end{align}
for $k\geq 0$. The support of $\mu$ is given by the spectrum of $\mathcal J_\infty$ \cite[p. 236]{reed1972methods} and is by construction infinite, but compact \cite[Theorem 1.3.7]{simon2011szego}. 

For $N\in \mathbb N$, the elements of the subspace $\mathcal E_N\subset \mathcal R$ with
\begin{align} \label{eq:Einfty}
\mathcal{E}_N = \big\{r=(b_1,a_1,\dots,b_N,0,0,\dots )\in \mathcal R \, \big| \, a_k>0 \text{ for all }1\leq k< N \big\}  
\end{align}
do not define Jacobi matrices for which the first unit vector is cyclic. However, we may for $r\in \mathcal E_N$ define the finite matrix 
\begin{align}\label{eq:tridiagonalfinite}
	\mathcal J_N(r) =\begin{pmatrix}
		b_1 & a_1 &  \\
		a_1 & b_2 & \ddots \\
		& \ddots & \ddots & a_{N-1}\\
		& & a_{N-1} & b_N
	\end{pmatrix} 
\end{align}
with positive subdiagonal. Then the spectral measure $\mu_N$ of $\mathcal J_N(r)$ is well-defined, finitely supported and satisfies the analogous moment relation as in \eqref{eq:spectralmeasuremoments}, that is, 
\begin{align}\label{eq:spectralmeasuremomentsfinite}
	\int x^k \, d\mu_N(x) = \langle e_1, \mathcal J_N(r)^k e_1 \rangle . 
\end{align}
Set $\mathcal E = \mathcal E_\infty \cup (\bigcup_{N=1}^\infty \mathcal E_N)$. We may then define a mapping
\begin{align} \label{eq:spectralmapping}
\varphi: \mathcal E \rightarrow \mathcal M_1^c(\mathbb{R}) , 
\end{align}
where $\mathcal M_1^c(\mathbb{R})$ is the set of compactly supported probability measures on $\mathbb R$, such that $\varphi(r) = \mu$ is the spectral measure satisfying \eqref{eq:spectralmeasuremoments} if $r\in \mathcal E_\infty$, and $\varphi(r)=\mu_N$ is the spectral measure satisfying \eqref{eq:spectralmeasuremomentsfinite} if $r\in \mathcal E_N$.  The following lemma shows that we can extend this mapping to the space $\mathcal R$, in which the LDP for the sequence $(r^{(n)})_{n\geq1}$ holds.

\begin{lem} \label{lem:continuous} 
We endow the space $\mathcal R$ with the trace of the product topology on $(\RR \times [0,\infty))^\NN$ and $\mathcal{M}^c_1(\RR)$ with the weak topology. Then the mapping $\varphi$ defined in \eqref{eq:spectralmapping} has a continuous extension $\tilde \varphi$ to $\mathcal R$. Moreover, for any $r\in \mathcal R$, there exists a unique $N\in \mathbb N\cup\{\infty\}$ and a unique $r_N\in \mathcal E_N$, such that $\tilde \varphi(r)=\varphi(r_N)$.  
\end{lem}
\begin{proof} 
	For $r = (b_1,a_1,b_2,a_2,...)\in \mathcal R$, we denote the infinite symmetric tridiagonal matrix build from $r$ by
	\begin{align}\label{eq:tridigonalinfinite}
		S_\infty(r)  = \begin{pmatrix}
			b_1 & a_1 &  \\
			a_1 & b_2 & a_2 \\
			& a_2 & \ddots & \ddots\\
			& & \ddots &  
		\end{pmatrix}.
	\end{align}
Since $S_\infty(r)$ is a self-adjoint bounded operator on $\ell^2$, there exists a unique spectral measure $\mu$ with compact support that satisfies the relation
\begin{align}
	\label{eq:zuordnungsvorschriftphi}
	\int x^k \, d\mu(x) = \langle e_1, S_\infty(r)^k e_1 \rangle
\end{align}
by the spectral theorem.
If $r \in \mathcal{E}$, we have 
\begin{align*}
	\langle e_1, S_\infty(r)^ke_1 \rangle 
	= \langle e_1, \mathcal{J}_N(r)^k e_1 \rangle
\end{align*}
for some $N \in \NN \cup \{\infty\}$. Thus setting $\tilde \varphi(r) = \mu$ with $\mu$ satisfying \eqref{eq:zuordnungsvorschriftphi} is indeed an extension of $\varphi$.

The function $\tilde \varphi$ is continuous: suppose $r_n \to r$ in $\mathcal R$, then $S_\infty(r_n)$ converges entrywise to $S_\infty(r)$, such that, according to \eqref{eq:zuordnungsvorschriftphi}, 
	\begin{align*}
		\int x^k \, d\tilde\varphi(r_n)  \xrightarrow[n\to \infty]{} \int x^k \, d\tilde\varphi(r) 
	\end{align*}
	for all $k\geq 1$. This yields the continuity of $\tilde\varphi$ because convergence of moments yields weak convergence if $\tilde\varphi(r)$ is compactly supported.
	
	For the remaining statement let $r = (b_1,a_1,b_2,a_2...) \in \mathcal{R}\setminus\mathcal{E}$. Then 
	\begin{align*}
		N := \inf\{k\in\NN \mid a_k = 0\}
	\end{align*}
	is finite and the matrix $S_\infty(r)$ has a block structure. More precisely,
	\begin{align}
		\label{eq:blockstructureJacobi}
		S_\infty(r) = \left(\begin{matrix}
			\mathcal{J}_N(r_N) & 0 \\
			0 &  S_\infty(\tilde r)
		\end{matrix}\right)
	\end{align}
	where
	\begin{align*}
		r_N := (b_1,a_1,...,b_N,0,0,...) \in \mathcal{E}_N 
		\quad \text{and} \quad
		\tilde r := (b_{N+1},a_{N+1},b_{N+2},a_{N+2},...) \in \mathcal{R}.
	\end{align*}
	The $k$-th power of $S_\infty(r)$ is
	\begin{align*}
		S_\infty(r)^k = \left(\begin{matrix}
			\mathcal{J}_N(r_N)^k & 0 \\
			0 &  S_\infty(\tilde r)^k
		\end{matrix}\right)
	\end{align*}
	so $\langle e_1, S_\infty(r)^k e_1 \rangle$ only depends on the entries of $\mathcal{J}_N(r_N)$. More precisely, we obtain
	\begin{align*}
		\langle e_1, S_\infty(r)^k e_1 \rangle = \langle e_1, \mathcal{J}_N(r_N)^k e_1 \rangle \quad \text{for all } k \in \NN.
	\end{align*}
	The uniqueness of $N$ and $r_N$ is immediate from the uniqueness of the Jacobi coefficients associated with the spectral measure \cite[Theorem 1.3.5,Theorem 1.3.8]{simon2011szego}.
\end{proof}
We are now ready to transfer the LDP from Jacobi coefficients to spectral measures. Recall that the LDP in Proposition \ref{prop:projectiveLDP} holds also in the space $\mathcal R$ and let $\tilde \varphi:\mathcal R\to \mathcal M_1^c(\mathbb R)$ be the continuous extension defined in Lemma \ref{lem:continuous}. For $r^{(n)}$ as defined in \eqref{eq:recsequence}, the measure $\mu_n=\varphi(r^{(n)}) = \tilde \varphi(r^{(n)})$ is then the spectral measure of $T_n$ as defined in \eqref{eq:tridiagonal}. By the work of \cite{Dumitriu_2002}, it has the same distribution as the measure given in \eqref{def:spectralmeasure}. The contraction principle then yields the LDP for the sequence $(\mu_n)_{n\geq 1}$ in the space $\mathcal M_1^c(\mathbb R)$ with speed $n$ and good rate function 
\begin{align} \label{eq:lastcontraction}
	\mathcal{I}(\mu) = \inf_{r \in \mathcal{R}:\tilde\varphi(r) = \mu} \mathcal{I}_5(r). 
\end{align}
If the value of the rate function $\mathcal I_5(r)$ is finite, the sequence $r$ of Jacobi coefficients is in $\ell^2$ and therefore $\mathcal{J}(r)$ is a Hilbert-Schmidt operator. A Hilbert-Schmidt operator is compact \cite[Theorem VI.22(e)]{reed1972methods} and its spectrum is discrete with  $0$ as the only possible accumulation point by the Riesz-Schauder Theorem  \cite[Theorem VI.15]{reed1972methods}. Therefore, $\mathcal I(\mu)$ is infinite if the support of $\mu$ is not countable. 

Suppose $\mu$ has compact and countable support and $|\operatorname{supp}(\mu)|=N$, with $N\in \mathbb N \cup\{\infty\}$. By Lemma \ref{lem:continuous}, there is a unique $r_N\in \mathcal E_N$ with $\varphi(r_N)=\mu$. If $r=(b_1,a_1,\dots)\in \mathcal R$ is any other sequence with $\tilde\varphi(r)=\mu$, the moment relations \eqref{eq:spectralmeasuremoments}, \eqref{eq:spectralmeasuremomentsfinite} imply that the first $2N-1$ entries $(b_1,a_1,\dots ,b_N)$ of $r$ agree with the first $2N-1$ entries of $r_N$. Since the subsequent entries of $r_N$ all vanish, we have
\begin{align*}
\mathcal I_5(r) - \mathcal I_5(r_N) = \frac{1}{2}\sum_{k=N+1}^\infty b_{k}^2 + \sum_{k\in J(r),k\geq N} (a_k^2 + \xi) \geq 0. 
\end{align*}
This implies
\begin{align}
\inf_{r \in \mathcal{R}:\tilde\varphi(r) = \mu} \mathcal{I}_5(r)= \mathcal I_5(r_N) .
\end{align}
The rate function $\mathcal I$ is therefore given by 
\begin{align}
\mathcal I(\mu) & = \mathcal I_5(r_N) = \frac{1}{2}\sum_{k=1}^N b_{k}^2 + \sum_{k=1}^{N-1} (a_k^2 + \xi) 
 = (N-1) \xi + \frac{1}{2} \tr (\mathcal J_N(r_N)^2),
\end{align}
where $\mathcal J_N(r_N)$ is the (finite or infinite) Jacobi matrix as in \eqref{eq:tridigonalinfinite} and \eqref{eq:tridiagonalfinite} with spectral measure $\mu$. Since 
\begin{align*}
\tr (\mathcal J_N(r_N)^2) = \sum_{j=1}^N \lambda_j(\mathcal J_N(r_N))^2 , 
\end{align*}
where $\lambda_j(\mathcal J_N(r_N))$ are the eigenvalues of $\mathcal J_N(r_N)$ and the support points of the spectral measure $\mu$, the rate function is as stated in Theorem \ref{thm:mainLDP}, restricted to the space $\mathcal{M}_1^c(\mathbb R)$ of compactly supported probability measures. If we extend $\mathcal I$ to $\mathcal M_1(\mathbb R)$, the set of all probability measures, by setting $\mathcal I(\mu)=\infty$ for $\mu$ not compactly supported, $\mathcal I$ is lower semicontinuous on $\mathcal M_1(\mathbb R)$. Therefore, the sequence $(\mu_n)_{n\geq 1}$ also satisfies the LDP in the larger space $\mathcal M_1(\mathbb R)$, which finishes the proof of Theorem \ref{thm:mainLDP}.\\

\medskip

\textbf{Acknowledgments:} 
This work was funded by the Deutsche Forschungsgemeinschaft (DFG, German Research Foundation) - Projektnummer 499508288

\bibliography{Literatur}
\bibliographystyle{abbrv}

\bigskip

{\footnotesize
	\noindent
	TU Dortmund, \\
	Fakult\"at f\"ur Mathematik, \\
	Vogelpothsweg 87, 
	44227 Dortmund, 
	Germany, \\
	{\tt helene.goetz@tu-dortmund.de}\\
	{\tt jan.nagel@tu-dortmund.de }
}
\end{document}